\documentclass[11pt,a4paper]{article} %report

\usepackage[usenames,dvipsnames]{color}
\usepackage{tikz}
\usepackage{lscape}

\usepackage[utf8]{inputenc}
\usepackage[T1]{fontenc}

\usepackage[english]{babel}

\usepackage{a4wide}

\usepackage{amssymb}
\usepackage{amsthm,amsmath}
\usepackage{amsfonts}
\usepackage{latexsym}
\usepackage{yhmath}

\usepackage{url} 
\usepackage[all]{xy} % diagramme
\usepackage{enumerate}
\usepackage{enumitem}

\usepackage{graphicx}

\theoremstyle{plain}
\newtheorem{thm}{Theorem}%[section]
\newtheorem*{thm*}{Theorem}
\newtheorem{lem}{Lemma}[section]
\newtheorem{cor}[lem]{Corollary}
\newtheorem{prop}[lem]{Proposition}

\theoremstyle{definition}
\newtheorem{df}[lem]{Definition}
\newtheorem{rem}[lem]{Remark}
\newtheorem{ex}[lem]{Example}

\newcommand{\bbR}{\mathbb{R}}

\newcommand{\bbZ}{\mathbb{Z}}

\newcommand{\bbH}{\mathbb{H}}
\newcommand{\bbO}{\mathbb{O}}

\newcommand{\R}{\mathbb{R}}
\newcommand{\Z}{\mathbb{Z}}
\newcommand{\C}{\mathbb{C}}

\newcommand{\cA}{{\mathcal{A}}}

\makeatletter

\@addtoreset{equation}{section}
\makeatother

\begin{document}

\title{Classification of the algebras $\bbO_{p,q}$}

\date{}

\author{Marie Kreusch and Sophie Morier-Genoud}

\maketitle

\begin{abstract}
We study a series of real nonassociative algebras $\bbO_{p,q}$ introduced in~\cite{MGO2011}.
These algebras have a natural $\Z_2^n$-grading, where $n=p+q$, and
they are characterized by a cubic form over the field $\Z_2$.
We establish all the possible isomorphisms between the algebras $\bbO_{p,q}$
preserving the structure of $\Z_2^n$-graded algebra.
The classification table of $\bbO_{p,q}$ is quite similar to that of 
the real Clifford algebras $\mathrm{Cl}_{p,q}$,
the main difference is that the algebras $\bbO_{n,0}$ and $\bbO_{0,n}$ are exceptional. 
\end{abstract}

%\tableofcontents

%%%%%%%%%%%%%%%%%%%%%%%%%%%%%%%%%%%%%%%%%%%%%%%%%%
%%%%%%%%%%%%%%%%%%%%%%%%%%%%%%%%%%%%%%%%%%%%%%%%%%
%1
\section{Introduction}
%%%%%%%%%%%%%%%%%%%%%%%%%%%%%%%%%%%%%%%%%%%%%%%%%%
%%%%%%%%%%%%%%%%%%%%%%%%%%%%%%%%%%%%%%%%%%%%%%%%%%

The series of noncommutative and nonassociative algebras, $\bbO_{p,q}$, over the field $\R$ of real numbers,
was introduced in~\cite{MGO2011}.
The algebras $\bbO_{p,q}$ generalize the classical $\Z_2^n$-graded algebras, where
$\Z_2^n=\Z/2\Z\times\cdots\times\Z/2\Z$ is the abelian group with $n=p+q$ generators of order $2$.
Noncommutativity and nonassociativity of the algebras are controlled by the grading,
within the framework developed in~\cite{AM1999,AM2002}. 
The main feature of the algebras $\bbO_{p,q}$ that distinguishes them from the 
algebras considered in~\cite{AM1999,AM2002}
(for instance, from the Cayley-Dickson algebras)
is that they are characterized by a {\it cubic form} on the vector space $\Z_2^n$ over 
the field~$\Z_2$ of two elements. 
The algebras $\bbO_{p,q}$ have applications to the classical Hurwitz
problem of sum of square identities and related problems, see~\cite{LMGO2011,MGO2013}.

The series of algebras $\bbO_{p,q}$ can be compared to the series of Clifford algebras, $\mathrm{Cl}_{p,q}$,
as well as to the Cayley-Dickson algebras, see the following diagram.

\medskip

\setlength{\unitlength}{3644sp}%
\begingroup\makeatletter\ifx\SetFigFont\undefined%
\gdef\SetFigFont#1#2#3#4#5{%
  \reset@font\fontsize{#1}{#2pt}%
  \fontfamily{#3}\fontseries{#4}\fontshape{#5}%
  \selectfont}%
\fi\endgroup%
\begin{picture}(6874,1966)(4951,-4685)
\put(7651,-4561){\makebox(0,0)[lb]{\smash{{\SetFigFont{12}{20.4}{\rmdefault}{\mddefault}{\updefault}{\color[rgb]{0,0,0}$\bbH$}%
}}}}
\thinlines
{\color[rgb]{0,0,0}\put(7988,-4448){\vector( 1, 0){900}}
}%
{\color[rgb]{0,0,0}\put(6638,-4448){\vector( 1, 0){900}}
}%
{\color[rgb]{0,0,0}\put(5288,-4448){\vector( 1, 0){900}}
}%
{\color[rgb]{0,0,0}\put(10418,-4493){\oval(360,360)[bl]}
\put(10418,-4516){\oval(360,360)[tl]}
\put(12000,-4493){\oval(360,360)[br]}
\put(12000,-4516){\oval(360,360)[tr]}
\put(10418,-4673){\line( 1, 0){1590}}
\put(10418,-4336){\line( 1, 0){1590}}
%\put(10238,-4493){\line( 0, 1){-23}}
%\put(11813,-4493){\line( 0, 1){-23}}
}%
{\color[rgb]{0,0,0}\put(8956,-3143){\oval(360,360)[bl]}
\put(8956,-2941){\oval(360,360)[tl]}
\put(9383,-3143){\oval(360,360)[br]}
\put(9383,-2941){\oval(360,360)[tr]}
\put(8956,-3323){\line( 1, 0){427}}
\put(8956,-2761){\line( 1, 0){427}}
\put(8776,-3143){\line( 0, 1){202}}
\put(9563,-3143){\line( 0, 1){202}}
}%
{\color[rgb]{0,0,0}\put(7268,-3143){\oval(360,360)[bl]}
\put(7268,-2941){\oval(360,360)[tl]}
\put(8258,-3143){\oval(360,360)[br]}
\put(8258,-2941){\oval(360,360)[tr]}
\put(7268,-3323){\line( 1, 0){990}}
\put(7268,-2761){\line( 1, 0){990}}
\put(7088,-3143){\line( 0, 1){202}}
\put(8438,-3143){\line( 0, 1){202}}
}%
{\color[rgb]{0,0,0}\put(9113,-4223){\vector( 0, 1){675}}
}%
{\color[rgb]{0,0,0}\put(7763,-4223){\vector( 0, 1){675}}
}%
\put(4951,-4561){\makebox(0,0)[lb]{\smash{{\SetFigFont{12}{20.4}{\rmdefault}{\mddefault}{\updefault}{\color[rgb]{0,0,0}$\R$}%
}}}}
\put(6301,-4561){\makebox(0,0)[lb]{\smash{{\SetFigFont{12}{20.4}{\rmdefault}{\mddefault}{\updefault}{\color[rgb]{0,0,0}$\C$}%
}}}}
\put(9001,-3211){\makebox(0,0)[lb]{\smash{{\SetFigFont{12}{16.8}{\rmdefault}{\mddefault}{\updefault}{\color[rgb]{0,0,0}$\bbO_{p,q}$}%
}}}}
\put(7651,-3211){\makebox(0,0)[lb]{\smash{{\SetFigFont{12}{16.8}{\rmdefault}{\mddefault}{\updefault}{\color[rgb]{0,0,0}$\mathrm{Cl}_{p,q}$}%
}}}}
\put(8888,-2986){\makebox(0,0)[lb]{\smash{{\SetFigFont{11}{16.8}{\rmdefault}{\mddefault}{\updefault}{\color[rgb]{0,0,0}Series}%
}}}}
\put(7201,-2986){\makebox(0,0)[lb]{\smash{{\SetFigFont{11}{16.8}{\rmdefault}{\mddefault}{\updefault}{\color[rgb]{0,0,0}Clifford Alg.}%
}}}}
\put(11026,-4561){\makebox(0,0)[lb]{\smash{{\SetFigFont{11}{16.8}{\rmdefault}{\mddefault}{\updefault}{\color[rgb]{0,0,0}Dickson Alg.}%
}}}}
\put(10351,-4561){\makebox(0,0)[lb]{\smash{{\SetFigFont{11}{16.8}{\rmdefault}{\mddefault}{\updefault}{\color[rgb]{0,0,0}Cayley}%
}}}}
\put(9001,-4561){\makebox(0,0)[lb]{\smash{{\SetFigFont{12}{20.4}{\rmdefault}{\mddefault}{\updefault}{\color[rgb]{0,0,0}$\bbO$}%
}}}}
{\color[rgb]{0,0,0}\put(9338,-4448){\vector( 1, 0){788}}
}%
\end{picture}%

The integer parameters $p,q$ refer to the signature of the set of generators of the algebras.
The problem of classification of $\bbO_{p,q}$ with respect to the signature $(p,q)$
was formulated in~\cite{MGO2011}.
In this paper, we answer this question.

The classification table of Clifford algebras has beautiful properties of symmetry and periodicity.
In this paper we establish quite similar properties for the algebras $\bbO_{p,q}$,
with the only difference: the algebras $\bbO_{n,0}$ and~$\bbO_{0,n}$ are exceptional.
Note that the Clifford algebras over $\R$ are very well understood; 
every algebra $\mathrm{Cl}_{p,q}$ is isomorphic to a matrix algebra over $\R, \C$ or~$\bbH$. 
The structure of the algebras $\bbO_{p,q}$ is more complicated and needs a further investigation.

The paper is organized as follows. 
In Section \ref{Sectiondefinition}, we provide three equivalent definitions of the algebras $\bbO_{p,q}$ 
and recall important notions related to them.  
In Section \ref{Sectionclassification}, our main result is settled. We establish isomorphisms between the algebras of the series that preserve the structure of $\mathbb{Z}_2^n$-graded algebra . 
We also present a table to illustrate the symmetry properties and to give an overview of the situation comparable to the Clifford case.

%%%%%%%%%%%%%%%%%%%%%%%%%%%%%%%%%%%%%%%%%%%%%%%%%%
%%%%%%%%%%%%%%%%%%%%%%%%%%%%%%%%%%%%%%%%%%%%%%%%%%
%2
\section{The algebras $\bbO_{p,q}$ generalizing the octonions}  \label{Sectiondefinition}
%%%%%%%%%%%%%%%%%%%%%%%%%%%%%%%%%%%%%%%%%%%%%%%%%%
%%%%%%%%%%%%%%%%%%%%%%%%%%%%%%%%%%%%%%%%%%%%%%%%%%
In this section, we recall the details about the algebras $\bbO_{p,q}$, see \cite{MGO2011}.  
We give three equivalent definitions of these algebras.
The first definition describes $\bbO_{p,q}$ as twisted group algebras over~$\Z_2^n$.
The second definition uses generators and relations.
The third way to describe $\bbO_{p,q}$ is to understand them as nonassociative deformations of Clifford algebras.

%%%%%%%%%%%%%%%%%%%%%%%%%%%%%%%%%%%%%%%%%%%%%%%%%%
\subsection{Twisted group algebras over $\bbZ_2^n$, the first definition of $\bbO_{p,q}$}
%%%%%%%%%%%%%%%%%%%%%%%%%%%%%%%%%%%%%%%%%%%%%%%%%%
We denote by $\bbZ_2$ the abelian group on two elements $\{0,1\}$.
Let  $f$ be any function from $\bbZ_2^n \times \bbZ_2^n $ to $\bbZ_2$. 
The twisted group algebra $\mathcal{A}= (\bbR\left[  \bbZ_2^n \right],f)$ is defined as the real linear space over the formal basis 
$\{u_x, x\in \bbZ_2^n\}$ together with the product given by
\begin{equation}\label{twistedproduct}
u_x \cdot u_y = (-1)^{f(x,y)} u_{x+y},
\end{equation}
for all $x,y\in \bbZ_2^n$.

\begin{ex}
(a)
The algebra of quaternions $\bbH$ 
($\,\simeq\mathrm{Cl}_{0,2}$), and more generally every Clifford algebra $\mathrm{Cl}_{p,q}$, 
can be realized as twisted group algebras over $\bbZ_2^n$, \cite{AM2002}.
Denote by $x=(x_1, \ldots,x_n)$ and $y=(y_1, \ldots,y_n)$ the elements in $\bbZ_2^n$.
The corresponding twisting function is bilinear:
\begin{equation}
\label{fclassic}
f _{\mathrm{Cl}_{p,q}} \left( x,y \right)=
\sum_{1 \leq i \leq j \leq n} x_i y_j + \sum_{1\leq{}i\leq{}p}x_i y_i \quad (n=p+q).
\end{equation}
In particular, $f _{\bbH} \left( x,y \right)=x_1y_1+x_1y_2+x_2y_2$.

(b)
The algebra of octonions $\bbO$ is a twisted group algebra over $\bbZ_2^3$
with the cubic twisting function:
\begin{equation}
\label{fclassic2}
f _{\bbO} \left( x,y \right)=
(x_1 x_2 y_3 +  x_1 y_2 x_3 + y_1 x_2 x_3 ) + \sum_{1 \leq i \leq j \leq 3} x_i y_j   ,
\end{equation}
see \cite{AM1999}.
\end{ex}

In general, a twisted group algebra is neither commutative nor associative. 
The defect of commutativity and associativity is
measured by a symmetric function $\beta:\bbZ_2^n \times \bbZ_2^n\to\bbZ_2$, 
and a function $\phi:\bbZ_2^n \times \bbZ_2^n \times \bbZ_2^n \to\bbZ_2$, respectively:
\begin{eqnarray} 
\label{comass}
u_x \cdot u_y &=& (-1)^{\beta(x,y)}\; u_y \cdot u_x \label{comass1} ,\\[4pt]
u_x \cdot ( u_y \cdot u_z)&  =& (-1)^{\phi(x,y,z)} \;(u_x \cdot  u_y) \cdot u_z  \label{comass2},
\end{eqnarray}
where explicitly
\begin{eqnarray}
\label{betaphi}
\beta(x,y) & =& f(x,y) + f(y,x) \label{betaphi1} ,\\[4pt]
\phi(x,y,z) &=&  f(y,z) + f(x+y,z) + f(x,y+z) + f(x,y) \label{betaphi2}  .
\end{eqnarray}
Such structures were studied in a more general setting, see e.g. \cite{AM1999, AEP2001}.

\begin{df}
\label{definition}\cite{MGO2011}
The  \emph{algebra $\bbO_{p,q}$} is the twisted group algebra 
$(\bbR\left[  \bbZ_2^n \right],f _{\bbO_{p,q}})$
with the twisting function 
\begin{equation}
\label{fO}
f _{\bbO_{p,q}} \left( x,y \right)= 
\sum_{1 \leq i < j < k \leq n} (x_i x_j y_k +  x_i y_j x_k + y_i x_j x_k ) + 
\sum_{1 \leq i \leq j \leq n} x_i y_j + \sum_{1\leq{}i\leq{}p}x_i y_i , 
\end{equation}
$n=p+q\geq 3$.
\end{df}

\begin{rem}
The algebra $\bbO_{0,3}$ is nothing but the classical algebra $\bbO$. 
Comparing the definitions of the twisting functions \eqref{fclassic}, \eqref{fclassic2} and \eqref{fO} 
gives a first viewpoint to understand 
the algebras $\bbO_{p,q}$ as a generalization of $\bbO$ in the same way as $\mathrm{Cl}_{p,q}$ generalize 
$\bbH$.
\end{rem}

For both series of algebras $\mathrm{Cl}_{p,q}$ and $\bbO_{p,q}$
the index $(p,q)$ is called the {\it signature}.

%%%%%%%%%%%%%%%%%%%%%%%%%%%%%%%%%%%%%%%%%%%%%%%%%%
\subsection{Twisted group algebras as graded algebras}
%%%%%%%%%%%%%%%%%%%%%%%%%%%%%%%%%%%%%%%%%%%%%%%%%%

By definition, every twisted group algebra $(\bbR\left[  \bbZ_2^n \right],f)$ is a natural graded algebra over the group $\bbZ_2^n$.
Consider the following basis elements of the abelian group $\bbZ_2^n$:
$$
e_i := (0, \ldots, 0, 1 , 0, \ldots , 0),
$$
where $1$ stands at i$^{th}$ position.
The corresponding homogeneous elements $u_i:=u_{e_i}$, $1\leq i \leq n$,
form a set of generators of the algebra $(\bbR\left[  \bbZ_2^n \right],f)$.
The {\it degree} of every generator is an element of $\Z_2^n$:
$$
\bar u_i := (0, \ldots, 0, 1 , 0, \ldots , 0)=e_i,
$$
where $1$ stands at i$^{th}$ position.
Let $u=u_{i_1}\cdots{}u_{i_k}$ be any monomial in the generators, define its degree (independently from the parenthesizing in the monomial) by
\begin{equation}
\label{GMon}
\bar u:=\bar u_{i_1}+\cdots+\bar u_{i_k},
\end{equation}
which is again an element of $\Z_2^n$.
The relations between the generators $u_i$ are entirely determined by the function $f$.

Note that the element $1:= u_{(0, \ldots , 0)}$ is the {\it unit} of the algebra.
Note also that the algebra is {\it graded-commutative} and {\it graded-associative}, as relations of type\eqref{comass}, \eqref{comass2} hold between homogeneous elements.

%%%%%%%%%%%%%%%%%%%%%%%%%%%%%%%%%%%%%%%%%%%%%%%%%%
\subsection{Generators and relations, the second definition of $\bbO_{p,q}$} \label{abstract algebras}
%%%%%%%%%%%%%%%%%%%%%%%%%%%%%%%%%%%%%%%%%%%%%%%%%%

We give here an alternative definition of the algebras $\bbO_{p,q}$ with generators and relations, 
see~\cite{MGO2011}.

In order to make this definition clear, we start with the classical abstract definition of the Clifford algebras.
The algebra $\mathrm{Cl}_{p,q}$
is the unique associative algebra generated by $n$ elements 
$v_1, \ldots, v_n$, where $n=p+q$, subject to the relations:
\begin{equation}\label{cliffrel}
\begin{array}{ccll}
v_i^2 &=& 
\left\{\begin{array}{cl}
1  &  \mbox{ if} \hspace{0.5cm}  1 \leq i \leq p,    \\
-1   & \mbox{ if} \hspace{0.5cm}  p+ 1 \leq i \leq n, \end{array}\right.\\[10pt]
v_{i} \cdot v_j&=&-v_{j} \cdot  v_i , & \mbox{ for all } \hspace{0.5cm}  i \neq j  \leq n. 
\end{array}
\end{equation}
Since $\mathrm{Cl}_{p,q}$ is associative, one of course has
$v_{i} \cdot  (v_j \cdot  v_k)=(v_i \cdot   v_j)\cdot   v_k$ for all $i , j , k\leq n$.

A similar definition of the algebras $\bbO_{p,q}$ consists in replacing
the associativity by a condition of graded-associativity.
First notice that there is a unique trilinear (or tri-addidive) alternate function
$
\phi:\Z_2^n\times\Z_2^n\times\Z_2^n\to\Z_2,
$
such that 
\begin{equation}
\label{defphi}
\phi(e_i,e_j,e_k)=1,
\end{equation}
for all $i\not= j\not= k$ in $\{1, \ldots ,n\}$.
This function is obviously {\it symmetric}, i.e.,
\begin{equation}
\label{Sic}
\phi(x,y,z)=\phi(x,z,y)=\ldots=\phi(z,y,x),
\end{equation}
for all $x,y,z\in\Z_2^n$.

%Conversely, the symmetry \eqref{Sic} condition implies that $\phi$ is trilinear 
%(Corollary 4.2 of~\cite{MGO2011}).

\begin{df} 
\label{definitionBis}\cite{MGO2011}
The \emph{algebra} $\bbO_{p,q}$ is the unique real unital algebra, 
generated by $n$ elements $ u_1, \ldots , u_n $, where $n=p+q$,
subject to the relations
\begin{equation}\label{genrel}
\begin{array}{ccll}
u_i^2 &=& 
\left\{\begin{array}{cl}
1  &  \mbox{ if} \hspace{0.5cm}  1 \leq i \leq p,    \\
-1   & \mbox{ if} \hspace{0.5cm}  p+ 1 \leq i \leq n, \end{array}\right.\\[10pt]
u_i \cdot u_j&=&-u_j \cdot  u_i , & \mbox{ for all } \hspace{0.5cm}  i \neq j  \leq n. 
\end{array}
\end{equation}
together with the graded associativity
\begin{equation}
\label{MonRel}
u\cdot(v\cdot w)=(-1)^{\phi(\bar u, \bar v, \bar w)}(u\cdot v)\cdot w,
\end{equation}
where $u,v,w$ are monomials, and $\phi$ the unique trilinear alternate form given by \eqref{defphi}.
\end{df}

In particular, the relation \eqref{MonRel} implies
$$
u_{i} \cdot  (u_j \cdot u_k)=-(u_i \cdot u_j) \cdot  u_k, 
$$
that is, the generators anti-associate with each other.
Note that the algebra $\bbO_{p,q}$ is graded-alternative, 
 i.e. $u\cdot(u\cdot v)=u^2\cdot v$, for all homogeneous elements $u$, $v$.

%%%%%%%%%%%%%%%%%%%%%%%%%%%%%%%%%%%%%%%%%%%%%%%%%%
\subsection{The generating cubic form} \label{abstract algebras}
%%%%%%%%%%%%%%%%%%%%%%%%%%%%%%%%%%%%%%%%%%%%%%%%%%
%%%%%%%%%
We now want to give a brief account on the theory developed in \cite{MGO2011}
and give the main ideas and main tools 
that we will need in the sequel. 

Consider an arbitrary twisted group algebra $\mathcal{A}= (\bbR\left[  \bbZ_2^n \right],f)$.
The multiplication rule of the algebra is encoded by $f$.
The structure of the algebra is encoded by $\beta$ and $\phi$, given in \eqref{comass1}--\eqref{comass2} and \eqref{betaphi1}--\eqref{betaphi2}. 

\begin{df}\label{defalpha}
A function $\alpha: \bbZ_2^n \longrightarrow \bbZ_2$ is called a \emph{generating function} 
for $\mathcal{A}$, if
$$
\begin{array}{llcl}
(i)&f(x,x)&=&\alpha(x),\\[6pt]
(ii)&\beta( x, y )  &=& \alpha (x+y) + \alpha (x) + \alpha (y) , \\[6pt]
(iii)&\phi( x, y, z )  &=& \alpha (x+y+z ) + \alpha (x+y) + \alpha (x+z) +  \alpha (y+z) + \alpha (x) + \alpha (y) + \alpha(z),
\end{array}
$$
for $x,y$ and $z$ in $\bbZ_2^n$.
\end{df}

One of the main results of \cite{MGO2011} is the following.

\begin{thm*} 
(i)
A twisted group algebra $\mathcal{A}$ has a generating function if and only if the function $\phi$ is 
symmetric as in \eqref{Sic}.

(ii)
The generating function $\alpha$ is a polynomial on $\Z_2^n$ of degree $\leq3$.
\end{thm*}

This result defines a class of twisted group algebras.
Every algebra of this class is characterized by a cubic form on $\Z_2^n$.
We will need the following converse result. 

\begin{prop}\label{uniqA} 
Given any cubic function $\alpha: \bbZ_2^n \longrightarrow \bbZ_2$ there exists a unique (up to isomorphism) twisted group algebra~$\mathcal{A}$ having $\alpha$ as a generating function.
\end{prop}
\begin{proof}
The existence of such twisted group algebra was proven in \cite{MGO2011}. 
There is a canonical way 
to construct a twisting function $f$ for which $\alpha$ will be a generating function.
Let us recall the construction. One writes explicitly $f(x,y)$ from the expression of $\alpha(x)$ 
according to the following procedure:
\begin{equation}
\label{Fakset}
\begin{array}{rcl}
x_ix_jx_k&\longmapsto&x_ix_jy_k+x_iy_jx_k+y_ix_jx_k,\\[4pt]
x_ix_j&\longmapsto&x_iy_j,\\[4pt]
x_i&\longmapsto&x_iy_i.
\end{array}
\end{equation}
where $1 \leq i<{}j<{}k \leq n$.

Assume that the twisted group algebra $\cA$ has a generating cubic form $\alpha$.
Let us now prove that $\cA$ is unique.

First, notice that $\alpha$ uniquely determines the relations of degree $2$ and $3$
between the generators $u_i$.
Indeed, $u_i^2=-1$ if and only if $\alpha$ contains the linear term $x_i$
(otherwise $u_i^2=1$);
$u_i$ and $u_j$ anticommute if and only if $\alpha$ contains the quadratic term $x_ix_j$
(otherwise, they commute);
$u_i \cdot (u_j \cdot u_k)=-(u_i \cdot u_j) \cdot u_k$ if and only if $\alpha$ contains the cubic term $x_ix_jx_k$
(otherwise, the generators associate).

The monomials
$$
u'_x=u_{i_1} \cdot (u_{i_2} \cdot (\;\cdots\; (u_{i_{l-1}} \cdot u_{i_{l}})\cdots)),
$$
for $x=e_{i_1}+e_{i_2}+\ldots+e_{i_l}$ with $i_1<i_2<\ldots<i_l$
form a (Poincar\'e-Birkhoff-Witt) basis of the algebra.
The product $u'_x \cdot u'_y$ of two such monomials
is equal to the monomial $\pm u'_{x+y}$.
The sign can be determined by
using only sequences of commutation and association between the generators, 
and the squares of the generators. 
Therefore the structure constants related to this basis are completely determined.
A twisting function $f'$ is deduced from
$$
u'_x \cdot u'_y=:(-1)^{f'(x,y)}u'_{x+y}.$$
\end{proof}

\begin{rem}
(i)
Two isomorphic algebras may have different generating functions.
For instance the algebra determined by 
$$
\alpha (x) = \sum_{1 \leq i < j \leq n-1}  x_i x_j x_n   + \sum_{1 \leq i  \leq j \leq n} x_i x_j + \sum_{1 \leq i \leq p }x_i.
$$
is isomorphic to $\bbO_{p,q}$, when $q>0$ (by sending $u_i$ to $u_{e_i+e_n}$, $1 \leq i  \leq n-1$, and $u_n$ to $u_n$.)

(ii)
In \cite{MGO2011}, the definition of generating function was given in a slightly more general situation of a quasialgebra $\cA$,
and therefore did not include the condition $(i)$ in the definition. 
The condition $(i)$ will also allow us to determine uniquely the algebra~$\cA$.
\end{rem}

%%%%%%%%%%%%%%%%%%%%%%%%%%%%%%%%%%%%%%%%%%%%%%%%%%
\subsection{The generating cubic form of the algebra $\bbO_{p,q}$}
%%%%%%%%%%%%%%%%%%%%%%%%%%%%%%%%%%%%%%%%%%%%%%%%%%

The algebras $\bbO_{p,q}$ (as well as the Clifford algebras) have generating cubic forms given by 
$$
\alpha_{p,q} (x) = f_{\bbO_{p,q}} (x,x) = \alpha_n (x) + \sum_{1\leq{}i\leq{}p} x_i ,
$$
where
$$
\alpha_n (x) = \sum_{1 \leq i < j < k \leq n}  x_i x_j x_k   + 
\sum_{1 \leq i  <j \leq n} x_i x_j + \sum_{1 \leq i \leq n }x_i.
$$

The cubic form  $\alpha_n$ is invariant under the action of
the group of permutations of the coordinates. 
Therefore, the value $\alpha_n (x)$ depends only on the weight 
(i.e. the number of nonzero components) of $x$ and can be easily computed. 
More precisely, for $x=(x_1, \ldots,x_n)$ in $\bbZ_2^n$, we denote the
{\it Hamming weight} of $x$ by
$$
|x|:=\#\{x_i\not=0\}.
$$
One has
\begin{eqnarray}\label{alphastand}
\alpha_n (x) = 
\left\{\begin{array}{cl}
0 , & \mbox{ if} \hspace{0.5cm}  |x| \equiv 0 \mbox{ mod } 4,   \\
1  ,& \mbox{ } \hspace{0.75cm}   \mbox{otherwise}.     
\end{array}\right.
\end{eqnarray}

The existence of a generating cubic form gives a way to 
distinguish the algebras $\mathrm{Cl}_{p,q}$ and $\bbO_{p,q}$
from other twisted group algebras.
For instance, the Cayley-Dickson algebras do not have generating cubic forms.
Let us also mention that
the generating cubic form is a very useful tool for the study of the algebras.

%%%%%%%%%%%%%%%%%%%%%%%%%%%%%%%%%%%%%%%%%%%%%%%%%%
\subsection{Nonassociative extension of Clifford algebras, the third definition of $\bbO_{p,q}$} \label{cliffext}
%%%%%%%%%%%%%%%%%%%%%%%%%%%%%%%%%%%%%%%%%%%%%%%%%%

Let us describe the third way to define the algebras $\bbO_{p,q}$.

Consider the subalgebra of the algebra $\bbO_{p,q}$
consisting in the elements of even degree.
The basis of this subalgebra is as follows 
$$
\{u_x: \;x\in (\Z_2)^n, |x|\equiv 0 \mod 2\}.
$$

\begin{prop}
The subalgebra of $\bbO_{p,q}$ of even elements
is isomorphic to the
Clifford algebra $\mathrm{Cl}_{p,q-1}$, for $q>0$.
\end{prop}

\begin{proof}
Indeed, consider the following elements: 
\[ 
v_i := u_{e_i+e_{n}} ,
\quad \mbox{for all }  1 \leq i \leq n -1.
\]
They generate all the even elements and they satisfy
\[\left\{\begin{array}{ccll}
v_i^2&=& 1, & \mbox{ if} \hspace{0.5cm}  1 \leq i \leq p , \\ 
v_i^2&=& -1, & \mbox{ if} \hspace{0.5cm}  p+1 \leq i \leq n-1,  \\ 
v_{i} \cdot v_j&=&-v_{j} \cdot  v_i ,& \mbox{ for all } \hspace{0.5cm}  i \neq j  \leq n-1 ,\\ 
v_{i} \cdot  (v_j \cdot  v_k)&=&(v_i \cdot   v_j)\cdot   v_k, & \mbox{ for all } \hspace{0.5cm} i , j , k  \leq n-1 . \\ 
\end{array}\right.\]
Linearity of the function $\phi$ implies that $v_i$ generate an associative algebra.
The above system of generators is
therefore a presentation for the real Clifford algebra $\mathrm{Cl}_{p,q-1}$.
\end{proof}

In other words, the algebra $\bbO_{p,q}$, for $q\not=0$, contains 
the following  {\it Clifford subalgebra}:
\begin{equation}
\label{CliffSub}
\mathrm{Cl}_{p,q-1}\subset\bbO_{p,q}.
\end{equation}
The algebra $\bbO_{p,q}$ thus can be viewed
as a nonassociative extension of this Clifford subalgebra  by one odd element, $u_n$,
anticommuting and antiassociating with all the elements of the Clifford algebra,
except for the unit.
As a vector space, 
$$
\bbO_{p,q}\simeq\mathrm{Cl}_{p,q-1}
\oplus \left( \mathrm{Cl}_{p,q-1}\cdot u_n \right).
$$
We will use the notation:
$\bbO_{p,q}\simeq\mathrm{Cl}_{p,q-1}\star u_n.$
This viewpoint is illustrated
by Figure \ref{HandO} in the case of quaternions and octonions. 

\begin{figure}[!h]
\begin{center}
\includegraphics[width=11cm]{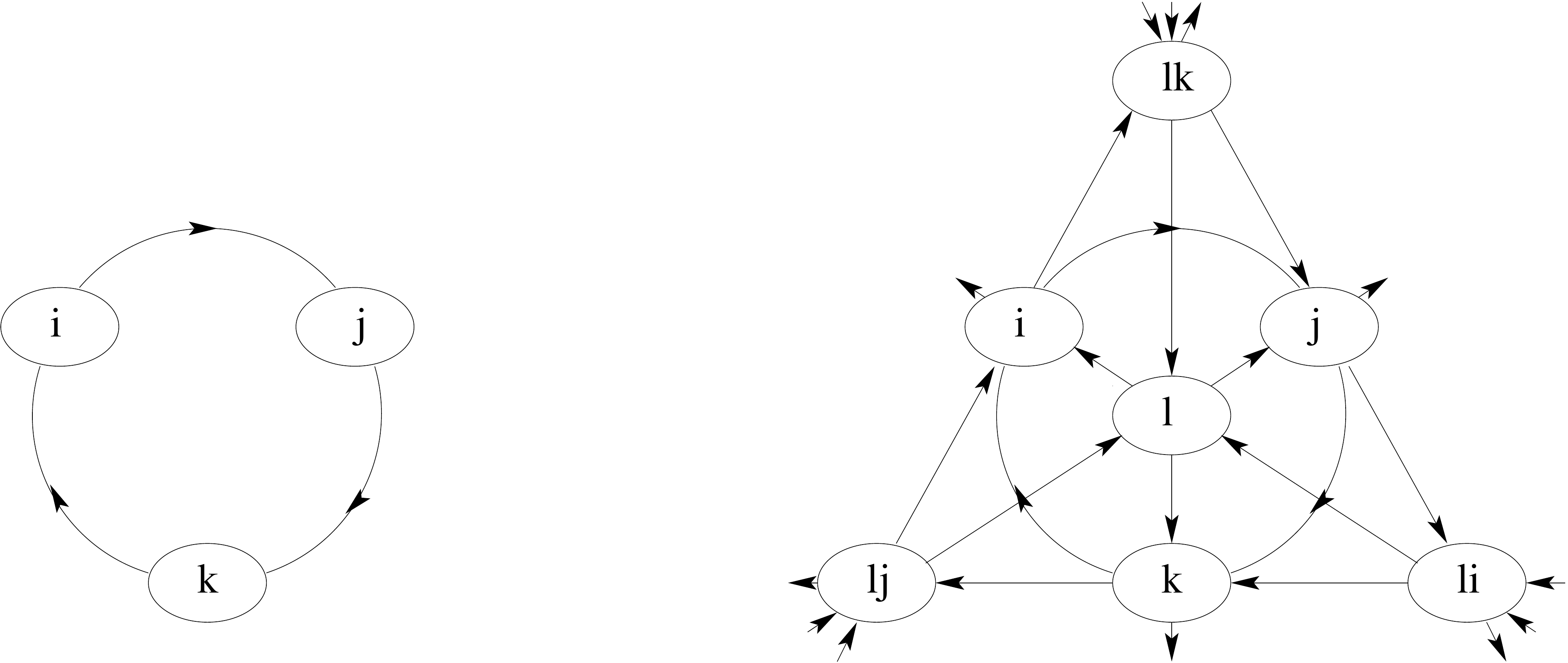}
\end{center}
\caption{Multiplication in the algebras of quaternions (left) and  of octonions (right)}
\label{HandO}
\end{figure}

%%%%%%%%%%%%%%%%%%%%%%%%%%%%%%%%%%%%%%%%%%%%%%%%%%
\subsection{Comparison of $\mathrm{Cl}_{p,q}$ and  $\bbO_{p,q}$}
%%%%%%%%%%%%%%%%%%%%%%%%%%%%%%%%%%%%%%%%%%%%%%%%%%
As mentioned in the previous section, there is a strong analogy between the algebras 
$\mathrm{Cl}_{p,q}$ and  $\bbO_{p,q}$
in terms of the various possible presentations of the algebras. 
Comparison of results on $\mathrm{Cl}_{p,q}$ and  $\bbO_{p,q}$ from \cite{MGO2011}, 
can be expressed in the following table.

\medskip
%\begin{figure}
\begin{tabular}{c|c}
 Clifford algebras & Algebras  $\bbO_{p,q}$\\[8pt]
   \hline
 $\mathrm{Cl}_{p,q}$ is simple if and only if:  &  $\bbO_{p,q}$  is simple if and only if: \\
$p+q \not\equiv 1 \mod 2$& $p+q \not\equiv 0 \mod 4$\\
or& or\\
$p+q \equiv 1 \mod 2$ and $p-q\equiv 3 \mod 4$ &$p+q \equiv 0 \mod 4$ and $p,q$ even\\[8pt]
   \hline
 if $\mathrm{Cl}_{p,q}$ is not simple, then  & if $\bbO_{p,q}$  is not simple, then \\
$\mathrm{Cl}_{p,q}\simeq \cA\oplus \cA $,& $\bbO_{p,q}\simeq \cA\oplus \cA$,\\
where& where\\
$\cA \simeq \mathrm{Cl}_{p-1,q}\simeq \mathrm{Cl}_{p,q-1}$&$\cA \simeq\bbO_{p-1,q}\simeq \bbO_{p,q-1}$\\[8pt]
\hline
 if $\mathrm{Cl}_{p,q}$ is simple and $p+q \equiv 1 \mod 2$,& if $\bbO_{p,q}$  is simple and $p+q \equiv 0 \mod 4$,  \\
then $\mathrm{Cl}_{p,q}\simeq \cA $,& then  $\bbO_{p,q}\simeq \cA$,\\
where& where\\
$\cA\simeq \mathrm{Cl}_{p,q-1}\otimes \C\simeq \mathrm{Cl}_{p-1,q}\otimes \C$&$\cA\simeq\bbO_{p,q-1}\otimes \C\simeq\bbO_{p-1,q}\otimes \C$\\
($\cA$ is a $\C$-algebra independent of $(p,q)$)&($\cA$ is a $\C$-algebra independent of $(p,q)$)\\[8pt]
\end{tabular}

%%%%%%%%%%%%%%%%%%%%%%%%%%%%%%%%%%%%%%%%%%%%%%%%%%
%%%%%%%%%%%%%%%%%%%%%%%%%%%%%%%%%%%%%%%%%%%%%%%%%%
%3
\section{Classification} \label{Sectionclassification}
%%%%%%%%%%%%%%%%%%%%%%%%%%%%%%%%%%%%%%%%%%%%%%%%%%
%%%%%%%%%%%%%%%%%%%%%%%%%%%%%%%%%%%%%%%%%%%%%%%%%%

In this section, we formulate and prove our main results. 
The problem we consider is to classify the isomorphisms
$\bbO_{p,q}\simeq \bbO_{p',q'}$
that preserve the structure of $\Z_2^n$-graded algebra.
This means, the isomorphisms sending 
homogeneous elements into homogeneous. 
We conjecture that whenever there is no isomorphism
preserving the structure of $\Z_2^n$-graded algebra
the algebras are not isomorphic.

In the end of the section, we present a table illustrating the results.

%%%%%%%%%%%%%%%%%%%%%%%%%%%%%%%%%%%%%%%%%%%%%%%%%%
\subsection{Statement of the main results} \label{res}
%%%%%%%%%%%%%%%%%%%%%%%%%%%%%%%%%%%%%%%%%%%%%%%%%%

Our main result is as follows.

\begin{thm}
\label{thmIso}
If $pq\not=0$, then there are the following isomorphisms of graded algebras:
\begin{enumerate}
\item[(i)] 
$\bbO_{p,q}\simeq \bbO_{q,p}\;$;
\item[(ii)] 
$\bbO_{p,q+4}\simeq \bbO_{p+4,q}\;$;
\item[(iii)] 
Every isomorphism between the algebras $\bbO_{p,q}$
preserving the structure of $\Z_2^n$-graded algebra is a combination of the above isomorphisms.
\item[(iv)] 
For $n\geq 5$, the algebras $\bbO_{n,0}$ and $\bbO_{0,n}$ are not isomorphic,
and are not isomorphic to any other algebras $\bbO_{p,q}$ with $p+q=n$.
\end{enumerate}
\end{thm}

Part (i) of the theorem gives a ``vertical symmetry'' with respect to $p-q=0$ and may be 
compared to the vertical symmetry with respect to $p-q=1$ in the case of Clifford algebras:
$$
\mathrm{Cl}_{p,q}\simeq \mathrm{Cl}_{q+1,p-1}.
$$
Part (ii) gives a periodicity modulo $4$ that also holds in the Clifford case:
$$
\mathrm{Cl}_{p+4,q} \simeq\mathrm{Cl}_{p,q+4}.
$$

Another way to formulate Theorem \ref{thmIso} is as follows.

\begin{cor}
\label{CorIso}
Assume that $p,p',q,q'\not=0$, then
$\bbO_{p,q}\simeq\bbO_{p',q'}$ if and only if
the corresponding Clifford subalgebras \eqref{CliffSub} are isomorphic;
the algebras $\bbO_{n,0}$ and $\bbO_{0,n}$ are exceptional.
\end{cor}

%%%%%%%%%%%%%%%%%%%%%%%%%%%%%%%%%%%%%%%%%%%%%%%%%%
\subsection{Construction of the isomorphisms} \label{subsectionisomorphisms}
%%%%%%%%%%%%%%%%%%%%%%%%%%%%%%%%%%%%%%%%%%%%%%%%%%

In this section, we establish a series of lemmas that provides all possible 
isomorphisms between the algebras $\bbO_{p,q}$.
Most of them have to be treated according to the congruence class of $n=p+q$ modulo 4.

We start with the algebras of small dimensions.

\begin{lem}
\label{ex1}
For $n=3,$ one has: 
\[\ \bbO_{3,0} \simeq \bbO_{2,1} \simeq  \bbO_{1,2} \not\simeq \bbO_{0,3}.\]
\end{lem}

\begin{proof}
To establish the isomorphism $\bbO_{3,0} \simeq \bbO_{2,1}$,
we consider the following coordinate transformation:
\[\left\{\begin{array}{ccl}
x'_1&=&x _1,  \\
x'_2&=&x_1 + x_3 , \\
x'_3&=&x_1 + x_2 + x_3.
\end{array}\right. 
\]
It is easy to check that $\alpha_{3,0}(x')=\alpha_{2,1}(x)$.

To establish the isomorphism $\bbO_{2,1} \simeq \bbO_{1,2}$, 
we can apply the above coordinate transformation twice.
We can also use the Clifford subalgebras. 
Write $\bbO_{2,1}\simeq \mathrm{Cl}_{2,0}\star u_3$ using
$$
v_1=u_{e_1+e_3}, \quad v_2=u_{e_2+e_3},
$$
as generators of $ \mathrm{Cl}_{2,0}$. 
Change these generators according to
$$
v'_1=v_1\cdot v_2, \quad v'_2=v_2.
$$
This gives two generators of $\mathrm{Cl}_{1,1}\simeq\mathrm{Cl}_{2,0}$ inside $\bbO_{2,1}$ 
that still anticommute and antiassociate with $u_3$.
Hence, $\bbO_{2,1}\simeq \mathrm{Cl}_{1,1}\star u_3\simeq \bbO_{1,2}$.

Let us prove that $\bbO_{0,3}$ is not (graded-)isomorphic to the other algebras. 
In the algebra $\bbO_{0,3}$ all seven homogeneous basis elements
different from the unit square to $-1$, 
whereas in $\bbO_{3,0},\bbO_{2,1}$ and $\bbO_{1,2}$ 
three elements square to $-1$ and four square to 1. 
Hence, there is no graded-isomorphism over $\R$. 
\end{proof}

\begin{rem}
Let us also mention that
$\bbO_{0,3}$ is isomorphic to the classical algebra of octonions,
whereas $\bbO_{3,0} \simeq \bbO_{2,1} \simeq  \bbO_{1,2}$ 
are isomorphic to the classical algebra of split octonions, see~\cite{MGO2011}.
\end{rem}

\begin{lem}
\label{ex2}
For $n=4,$ one has: 
\[\ \bbO_{4,0} \simeq \bbO_{2,2} \hspace{1cm} \mbox{and}\hspace{1cm}  \bbO_{3,1} \simeq \bbO_{1,3} .\]
\end{lem}

\begin{proof}
Consider the original basis $\{u_x\}_{x\in \Z_2^4}$ of $\bbO_{4,0}$, and as usual we simply denote by $u_i=u_{e_i}$. The following set of elements
$$
\left\{\begin{array}{ccl}
u'_1&=&u_{e_1 + e_4}  , \\
u'_2&=&u_{e_2 + e_4} ,  \\
u'_3&=&u_3  ,  \\
u'_4&=&u_4, 
\end{array}\right.
$$
form a system of generators, that anticommute and antiassociate. The signature of this set is $(2,2)$.
Hence, $\bbO_{4,0}$ is isomorphic to
$\bbO_{2,2}$. 

Similarly, consider the following change of generators
$$
%\quad
%\left\{\begin{array}{ccl}
%u'_1&=&u_{e_1 + e_4}  , \\
%u'_2&=&u_{e_2 + e_4}  ,\\
%u'_3&=&u_{e_1 + e_3 + e_4}   , \\
%u'_4&=&u_{e_1 + e_2 + e_4},
%\end{array}\right.
%\quad\mbox{or}\quad
\left\{\begin{array}{ccl}
u'_1&=&u_{e_2 + e_3 + e_4}  , \\
u'_2&=&u_{e_1 + e_3 + e_4}  ,\\
u'_3&=&u_{e_1 + e_2 + e_4}   , \\
u'_4&=&u_{e_1 + e_2 + e_3}.
\end{array}\right.$$
The new generators anticommute and antiassociate and have signature $(3,1)$ if the initial ones had signature $(1,3)$ and vice versa. Hence the isomorphism $ \bbO_{3,1} \simeq \bbO_{1,3}$.
\end{proof}

Note that in order to determine the (anti)commutativity or the (anti)associativity between elements
one can use the formulas (ii) and (iii) of Definition \ref{defalpha} and evaluate them using the standard form $\alpha_n$  \eqref{alphastand} (since linear terms vanish in the formulas (ii) and (iii)) which depends only on the Hamming weight of the elements.

 We now extend our method to higher dimensional algebras.

\begin{lem} 
If $p+q  = 4k$  and $p,q $ are even, then
\[
\begin{array}{cccclc}
\bbO_{p,q} \simeq\bbO_{p,q-1} \oplus  \bbO_{p-1,q} , &&& \quad \mbox{ if} &p \geq 2 \mbox{ and }   q \geq 2,   \\
\end{array}
\]
and
\[\ \bbO_{4k,0} \simeq \bbO_{4k-1,0} \oplus  \bbO_{4k-1,0} \hspace{0.2cm}, \hspace{1cm} \bbO_{0,4k} \simeq \bbO_{0,4k-1} \oplus  \bbO_{0,4k-1}    . \]
\end{lem}

\noindent
This statement is proved in~\cite{MGO2011} in the complex 
case (see Theorem 3, p. 100).
The proof in the real case is identically the same.  

The next four lemmas give the list of isomorphisms with respect to the congruence class $p+q$ modulo $4$.

\begin{lem} 
\label{package4}
(i)
If  $p+q=4k$ and $p,q$  are even, then
\[
\begin{array}{cccclc}
\bbO_{p,q} \simeq \bbO_{p+4,q-4}, &&& \quad \mbox{ if} &p \geq 2 \mbox{ and }   q \geq 6.   \\
\end{array}
\]

(ii)
If   $p+q=4k$ and $p,q$ are odd, then
\[
\begin{array}{cccclc}
\bbO_{p,q} \simeq \bbO_{p+2,q-2}, &&& \quad \mbox{ if} &p \geq 1 \mbox{ and }   q \geq 3 .  \\
\end{array}
\]
\end{lem}

\begin{proof}
We define a new set of generators splitted in blocks of four elements
\[\left\{\begin{array}{ccc}
u'_{4i+1}&=&u_{e_{4i+2} +e_{4i+3} +e_{4i+4}} , \\
u'_{4i+2}&=&u_{e_{4i+3} + e_{4i+4} +e_{4i+1} } , \\
u'_{4i+3}&=&u_{e_{4i+4} + e_{4i+1} +e _{4i+2}}  , \\
u'_{4i+4}&=&u_{e_{4i+1} + e_{4i+2} + e_{4i+3} } . 
\end{array}\right.\]
for every $i \in \{ 0, \ldots, k-1\} $.
This new generators are  still anticommuting and antiassociating.
The signature in each block remains unchanged if the initial one was $(4,0)$, $(2,2)$ or $(0,4)$ and changes from $(3,1)$ to $(1,3)$ and vice versa. 
Therefore one can organize the generators in the blocks in order to obtain the desired isomorphisms.
\end{proof}

\begin{lem} \label{shift268}
If $p+q=4k+1$, then
\[
\begin{array}{cccclcl}
\bbO_{p,q} \simeq \bbO_{p-4,q+4}, &&& \quad \mbox{ if} &p-4 \geq 1  \mbox{  and  }    q \geq 1, \\[4pt]
\bbO_{p,q} \simeq \bbO_{p+1,q-1} ,&&& \quad \mbox{ if}   & p \geq 1 \mbox{  and   }    q-1 \geq 1, 
& p \text{ even}.
\end{array}
\]
\end{lem}

\begin{proof}
Consider the same change of the first $4k$ generators as 
in Lemma \ref{package4}, and change also the last generator as follows
\[\left\{\begin{array}{ccc}
u'_{4i+1}&=&u_{e_{4i+2} +e_{4i+3} +e_{4i+4}} , \\
u'_{4i+2}&=&u_{e_{4i+3} + e_{4i+4} +e_{4i+1} }, \\
u'_{4i+3}&=&u_{e_{4i+4} + e_{4i+1} +e _{4i+2}} ,\\
u'_{4i+4}&=&u_{e_{4i+1} + e_{4i+2} + e_{4i+3} },
\end{array}\right.
\quad \mbox{and} \quad u'_{n}=u_{z},
\]
where $i \in \{ 0, \ldots, k-1\} $,
and $z$ denotes the element of maximal weight:
$$
z= (1, \ldots, 1) \in \mathbb{Z}_2^{n}.
$$
One then checks directly that these changes give the desired isomorphisms. 
\end{proof}

\begin{lem}
(i)
If $p+q=4k+2$ and $p,q$ are odd, then 
\[
\begin{array}{cccclc}
\bbO_{p,q} \simeq \bbO_{p+4,q-4}, &&& \quad \mbox{ if} & p \geq 1 \mbox{ and }   q \geq 5.  \\
\end{array}
\]

(ii)
If $p+q=4k+2$ and $p,q$ are even, then
\[
\begin{array}{cccclc}
\bbO_{p,q} \simeq \bbO_{p-2,q+2}, &&& \quad \mbox{ if} &p \geq4  \mbox{ and }   q \geq 2 . \\
\end{array}
\]
\end{lem}

\begin{proof}
Consider the following change of the generators 
\[\left\{\begin{array}{cclc}
u'_i&=&u_{e_i+w_2}, & \mbox{ if} \hspace{0.5cm}  1 \leq i \leq 4  , \\[4pt]
u'_i&=&u_i,   & \mbox{ otherwise} ,
\end{array}\right.\]
where $w_2= (0,0,0,0,1, \ldots, 1)$ is an element in $\mathbb{Z}_2^{n}$. 
%This change shifts the signature of $4$, resp. $8$, depending on the parity of $p$ and $q$. 
\end{proof}

\begin{lem}
If $p+q=4k+3$, then
\[
\begin{array}{cccclcl}
\bbO_{p,q} \simeq \bbO_{p-4,q+4} , &&& \quad \mbox{ if} &  p-4  \geq  1  \mbox{  and  }   q  \geq 1 ,  \\[4pt]
\bbO_{p,q} \simeq \bbO_{p-1,q+1}, &&& \quad \mbox{ if}   &    p  \geq 1   \mbox{  and  }    q  \geq 1, \; p\text{ even}  . 
\end{array}
\]
\end{lem}

\begin{proof}
The isomorphisms are given by considering the following change of the generators 
\[\left\{\begin{array}{ccll}
u'_i&=&u_{e_i + w_3}, & \mbox{ if} \hspace{0.5cm}  1 \leq i \leq 4 ,  \\
u'_{i}&=&u_i , &  \mbox{ if} \hspace{0.5cm}  1 \leq i \leq  p+q-2 ,\\
u'_{p+q-1}&=&u_{w_1}  ,&   \\
u'_{p+q}&=&u_{w_2} ,  &   
\end{array}\right.\]
where $w_1= (1,\ldots, 1,0)$, $w_2= (0,0,0,0,1, \ldots, 1)$ and $w_3= (0,0,0,0,1, \ldots, 1, 0)$ are elements in $\mathbb{Z}_2^{n}$. 
 \end{proof}

The combination of all the above lemmas implies Theorem \ref{thmIso}, part $(i)$ and part $(ii)$.

\begin{rem}
The isomorphisms 
$
\bbO_{p,q} \simeq \bbO_{p-1,q+1},
$
in the case $n=p+q=4k+3$, $p$ even, can be established 
using connection to Clifford algebras. 
Indeed,
consider the Clifford subalgebra $\mathrm{Cl}_{p,q-1}\subset\bbO_{p,q}$, 
see Section \ref{abstract algebras},
with the generators $v_i$. 
For $p+q = 4k+3$ and $p$ even, the classical isomorphism is
\[ 
\mathrm{Cl}_{p,q-1} \simeq \mathrm{Cl}_{p-1,q}.
\]
This isomorphism can be given by the change of variables on generators:
\[\left\{\begin{array}{cclc}
v'_1&=&v_1 \cdots  v_{n-1} , &\\[4pt]
v'_i&=&v_i ,& \mbox{ if} \hspace{0.5cm}  1 \leq i \leq p+q-1   .
\end{array}\right.\]
Add the generator $u_{n}$ of weight $1$ 
that anticommutes and antiassociates with $v'_i$, 
one obtains the algebra $\bbO_{p-1,q+1}$. 
Hence the result.
\end{rem}

The above isomorphism can be illustrated by the following diagram:   
\begin{center}
\begin{tabular}{ccccccc}
$\mathrm{Cl}_{p,q-1}$ $\subset$ &$  < v_1, \ldots , v_{n-1} >  \hspace{0.1cm}\star \hspace{0.1cm} < u_{n} >  $  && $\simeq$& & $ < v'_1, \ldots , v'_{n-1} >  \hspace{0.1cm}\star  \hspace{0.1cm} < u_{n} > $  &$\supset$ $\mathrm{Cl}_{p-1,q}$  \\
&\reflectbox{\rotatebox[origin=c]{270}{$\simeq$}}  && && \reflectbox{\rotatebox[origin=c]{90}{$\simeq$}} &\\
& $\bbO_{p,q}$  &&&& $\bbO_{p-1,q+1}$ &
\end{tabular}
\end{center}

%%%%%%%%%%%%%%%%%%%%%%%%%%%%%%%%%%%%%%%%%%%%%%%%%%
\subsection{Criterion of non isomorphism}
%%%%%%%%%%%%%%%%%%%%%%%%%%%%%%%%%%%%%%%%%%%%%%%%%%

In order to prove Theorem \ref{thmIso}, Parts $(iii)$ and $(iv)$,
we will define an {\it invariant} of the algebras:
we count how many homogeneous basis elements square to $-1$.
This invariant will be called the {\it statistics}.

\begin{df}
Define the \emph{statistics} of the algebra $\bbO_{p,q}$ by 
\[ 
s(p,q) :=  \# \left\{ x \in \bbZ_2^n : \alpha_{p,q}(x) = 1  \right\} .
\]
\end{df}

Clearly, 
$$
0\leq{}s(p,q)\leq2^n,
$$
where $n=p+q$.  

\begin{lem}
\label{EtOui!}
The number $s(p,q)$ is invariant with respect to an isomorphism 
preserving the structure of $\Z_2^n$-graded algebra.
\end{lem}

\begin{proof}
An isomorphism preserving the structure of $\Z_2^n$-graded algebra
 sends  homogeneous basis elements to homogeneous.
Since  homogeneous components of $\bbO_{p,q}$ are of dimension $1$,
this means that the isomorphism multiplies every basis element $u$ by a constant.
Hence it preserves the sign of $u^2$.
\end{proof}

Our goal is to show that
$$
s \left( n,0 \right) < s \left( p,q \right) < s \left( 0, n \right),
$$
for all $p,q\not=0$, where $p+q=n$.

\begin{lem}
Let $n \geq 3$, the algebras $\mathbb{O}_{0,n}$ and $\mathbb{O}_{n,0}$ are not isomorphic.
\end{lem}
\begin{proof}
We compute the statistics for these algebras:
$$
s \left( n, 0 \right)  =  \sum \limits_{\underset{4i+2 \leq n}{i=0}}^{k} \binom {n} {4i+2}, \qquad
s \left( 0,n \right)  =  2^n - \sum\limits_{\underset{4i \leq n}{i=0}}^{k} \binom {n} {4i}.
$$
Clearly, $s \left( n,0 \right)\not=s \left( 0, n \right) $ since
$$
\sum\limits_{\underset{4i \leq n}{i=0}}^{k} \binom {n} {4i}+
\sum \limits_{\underset{4i+2 \leq n}{i=0}}^{k} \binom {n} {4i+2}
 \quad < \quad
 \sum_{i=0}^{n} \binom {n} {i}=2^n.
$$
\end{proof}

\begin{lem}
Let $n \geq 5$, the algebras $\mathbb{O}_{n,0}$ and $\mathbb{O}_{0,n}$ are not isomorphic to any algebras $\bbO_{p,q}$. Furthermore, we have the different cases. \\
If $n= 4k$, then 
\[  \bbO_{n-1, 1} \not\simeq \bbO_{n-2, 2}   \not\simeq \bbO_{n-4, 4}  .  \]
If $n= 4k+1$, then 
\[ \bbO_{n-1, 1} \not\simeq \bbO_{n-2, 2}  . \]
If $n= 4k+2$, then 
\[ \bbO_{n-1, 1} \not\simeq \bbO_{n-2, 2}    \not\simeq \bbO_{n-3, 3}.  \]
If $n= 4k+3$, then 
\[ \bbO_{n-1, 1}    \not\simeq \bbO_{n-3, 3}.  \]

\end{lem}

\begin{proof}
Let $n= 4k$, then 
\begin{align*}
s \left( n-1, 1 \right) & = \sum_{i=0}^{k-1} \binom {n-1} {4i} +2  \binom {n-1} {4i+2} + \binom {n-1} {4i+3}, \\ 
    s \left( n-2, 2 \right) &=   \sum_{i=0}^{k-1} 3 \binom {n-2} {4i} +3   \binom {n-2} {4i+2} +   \sum_{i=0}^{k-2}  2 \binom {n-2} {4i+3},  \\
s \left( n-4, 4 \right) &=   \sum_{i=0}^{k} 14 \binom {n-4} {4i} +
\sum_{i=0}^{k-2}  4   \binom {n-4} {4i+1} +  10   \binom {n-4} {4i+2} +  4 \binom {n-4} {4i+3} .
\end{align*}
Let $n= 4k + 1$, then 
\begin{align*}
 s \left( n-1, 1 \right)& =   \sum_{i=0}^{k-1} \binom {n-1} {4i} +2  \binom {n-1} {4i+2} + \binom {n-1} {4i+3} + \binom {n-1} {n-1},  \\
s \left( n-2, 2 \right) &=   \sum_{i=0}^{k-1} 3 \binom {n-2} {4i} +3   \binom {n-2} {4i+2} +  2 \binom {n-2} {4i+3}  . 
\end{align*}
Let $n= 4k + 2$, then 
\begin{align*}
s \left( n-1, 1 \right) &=   \sum_{i=0}^{k-1} \binom {n-1} {4i} +2  \binom {n-1} {4i+2} + \binom {n-1} {4i+3} + \binom {n-1} {n-2} ,\\
s \left( n-2, 2 \right) &=   \sum_{i=0}^{k} 3 \binom {n-2} {4i} +   \sum_{i=0}^{k-1}   3   \binom {n-2} {4i+2} + 2 \binom {n-2} {4i+3}, \\
s \left( n-3, 3 \right) &=   \sum_{i=0}^{k-1} 7  \binom {n-3} {4i} +  \binom {n-3} {4i+1}  + 5  \binom {n-3} {4i+2} + 3 \binom {n-3} {4i+3}.  
\end{align*}
Let $n= 4k + 3$, then 
\begin{align*}
s \left( n-1, 1 \right)& =   \sum_{i=0}^{k} \binom {n-1} {4i} +2  \binom {n-1} {4i+2} +  \sum_{i=0}^{k-1}  \binom {n-1} {4i+3}  ,\\
s \left( n-3, 3 \right) &=   \sum_{i=0}^{k} 7  \binom {n-3} {4i} + \sum_{i=0}^{k-1}   \binom {n-3} {4i+1}  + 5  \binom {n-3} {4i+2} + 3 \binom {n-3} {4i+3} .
\end{align*}

For $n\geq 5$, all the above the statistics are distinct and are strictly bounded by $s(n,0)$, the minimum, and $s(0,n)$, the maximum.
\end{proof}

Considering the symmetry and the periodicity, the above lemmas imply Theorem \ref{thmIso}, 
Parts $(iii)$ and $(iv)$.

%%%%%%%%%%%%%%%%%%%%%%%%%%%%%%%%%%%%%%%%%%%%%%%%%%
\subsection{Table of $\mathbb{O}_{p,q}$}
%%%%%%%%%%%%%%%%%%%%%%%%%%%%%%%%%%%%%%%%%%%%%%%%%%

The following table allows one to visualize the classes of algebras $\mathbb{O}_{p,q}$ and the  corresponding statistics.

  \begin{landscape}
\begin{center}

\begin{tikzpicture}  [scale=0.9]
%\draw [very thin, gray] (0,0) grid(25,10) ; 
%C'est pour dessiner les petits cadres

%ligne12
\draw (0.5,0.7) node{   \textcolor{Goldenrod}{$\mathbb{O}_{12,0}$}};
\draw (0.5,0.3) node{\tiny{$1056$}};

\draw (2.5,0.7) node{\textcolor{VioletRed}{$\mathbb{O}_{11,1}$}};
\draw (2.5,0.3) node{\tiny{$2048$}};

\draw (4.5,0.7) node{\textcolor{PineGreen}{$\mathbb{O}_{10,2}$}};
\draw (4.5,0.3) node{\tiny{$2016$}};

\draw (6.5,0.7) node{\textcolor{VioletRed}{$\mathbb{O}_{9,3}$}};
\draw (6.5,0.3) node{\tiny{$2048$}};

\draw (8.5,0.7) node{\textcolor{SpringGreen}{$\mathbb{O}_{8,4}$}};
\draw (8.5,0.3) node{\tiny{$2080$}};

\draw (10.5,0.7) node{\textcolor{VioletRed}{$\mathbb{O}_{7,5}$}};
\draw (10.5,0.3) node{\tiny{$2048$}};

\draw (12.5,0.7) node{\textcolor{PineGreen}{$\mathbb{O}_{6,6}$}};
\draw (12.5,0.3) node{\tiny{$2016$}};

\draw (14.5,0.7) node{\textcolor{VioletRed}{$\mathbb{O}_{5,7}$}};
\draw (14.5,0.3) node{\tiny{$2048$}};

\draw (16.5,0.7) node{\textcolor{SpringGreen}{$\mathbb{O}_{4,8}$}};
\draw (16.5,0.3) node{\tiny{$2080$}};

\draw (18.5,0.7) node{\textcolor{VioletRed}{$\mathbb{O}_{3,9}$}};
\draw (18.5,0.3) node{\tiny{$2048$}};

\draw (20.5,0.7) node{\textcolor{PineGreen}{$\mathbb{O}_{2,10}$}};
\draw (20.5,0.3) node{\tiny{$2016$}};

\draw (22.5,0.7) node{\textcolor{VioletRed}{$\mathbb{O}_{1,11}$}};
\draw (22.5,0.3) node{\tiny{$2048$}};

\draw (24.5,0.7) node{  \textcolor{red}{$\mathbb{O}_{0,12}$}};
\draw (24.5,0.3) node{\tiny{$3104$}};

%ligne11
\draw (1.5,1.7) node{   \textcolor{Goldenrod}{$\mathbb{O}_{11,0}$}};
\draw (1.5,1.3) node{\tiny{$528$}};

\draw (3.5,1.7) node{\textcolor{blue}{ $\mathbb{O}_{10,1}$}};
\draw (3.5,1.3) node{\tiny{$1008$}};

\draw (5.5,1.7) node{\textcolor{blue}{ $\mathbb{O}_{9,2}$}};
\draw (5.5,1.3) node{\tiny{$1008$}};

\draw (7.5,1.7) node{ \textcolor{SkyBlue}{$\mathbb{O}_{8,3}$}};
\draw (7.5,1.3) node{\tiny{$1040$}};

\draw (9.5,1.7) node{ \textcolor{SkyBlue}{$\mathbb{O}_{7,4}$}};
\draw (9.5,1.3) node{\tiny{$1040$}};

\draw (11.5,1.7) node{ \textcolor{blue}{$\mathbb{O}_{6,5}$}};
\draw (11.5,1.3) node{\tiny{$1008$}};

\draw (13.5,1.7) node{\textcolor{blue}{ $\mathbb{O}_{5,6}$}};
\draw (13.5,1.3) node{\tiny{$1008$}};

\draw (15.5,1.7) node{\textcolor{SkyBlue}{ $\mathbb{O}_{4,7}$}};
\draw (15.5,1.3) node{\tiny{$1040$}};

\draw (17.5,1.7) node{ \textcolor{SkyBlue}{$\mathbb{O}_{3,8}$}};
\draw (17.5,1.3) node{\tiny{$1040$}};

\draw (19.5,1.7) node{ \textcolor{blue}{$\mathbb{O}_{2,9}$}};
\draw (19.5,1.3) node{\tiny{$1008$}};

\draw (21.5,1.7) node{\textcolor{blue}{ $\mathbb{O}_{1,10}$}};
\draw (21.5,1.3) node{\tiny{$1008$}};

\draw (23.5,1.7) node{   \textcolor{red}{$\mathbb{O}_{0,11}$}};
\draw (23.5,1.3) node{\tiny{$1552$}};

%ligne10
\draw (2.5,2.7) node{  \textcolor{Goldenrod}{$\mathbb{O}_{10,0}$}};
\draw (2.5,2.3) node{\tiny{$256$}};

\draw (4.5,2.7) node{\textcolor{PineGreen}{$\mathbb{O}_{9,1}$}};
\draw (4.5,2.3) node{\tiny{$496$}};

\draw (6.5,2.7) node{\textcolor{VioletRed}{$\mathbb{O}_{8,2}$}};
\draw (6.5,2.3) node{\tiny{$512$}};

\draw (8.5,2.7) node{\textcolor{SpringGreen}{$\mathbb{O}_{7,3}$}};
\draw (8.5,2.3) node{\tiny{$528$}};

\draw (10.5,2.7) node{\textcolor{VioletRed}{$\mathbb{O}_{6,4}$}};
\draw (10.5,2.3) node{\tiny{$512$}};

\draw (12.5,2.7) node{\textcolor{PineGreen}{$\mathbb{O}_{5,5}$}};
\draw (12.5,2.3) node{\tiny{$496$}};

\draw (14.5,2.7) node{\textcolor{VioletRed}{$\mathbb{O}_{4,6}$}};
\draw (14.5,2.3) node{\tiny{$512$}};

\draw (16.5,2.7) node{\textcolor{SpringGreen}{$\mathbb{O}_{3,7}$}};
\draw (16.5,2.3) node{\tiny{$528$}};

\draw (18.5,2.7) node{\textcolor{VioletRed}{$\mathbb{O}_{2,8}$}};
\draw (18.5,2.3) node{\tiny{$512$}};

\draw (20.5,2.7) node{\textcolor{PineGreen}{$\mathbb{O}_{1,9}$}};
\draw (20.5,2.3) node{\tiny{$496$}};

\draw (22.5,2.7) node{  \textcolor{red}{$\mathbb{O}_{0,10}$}};
\draw (22.5,2.3) node{\tiny{$768$}};

%ligne9
\draw (3.5,3.7) node{   \textcolor{Goldenrod}{$\mathbb{O}_{9,0}$}};
\draw (3.5,3.3) node{\tiny{$120$}};

\draw (5.5,3.7) node{\textcolor{blue}{ $\mathbb{O}_{8,1}$}};
\draw (5.5,3.3) node{\tiny{$248$}};

\draw (7.5,3.7) node{ \textcolor{SkyBlue}{$\mathbb{O}_{7,2}$}};
\draw (7.5,3.3) node{\tiny{$264$}};

\draw (9.5,3.7) node{ \textcolor{SkyBlue}{$\mathbb{O}_{6,3}$}};
\draw (9.5,3.3) node{\tiny{$264$}};

\draw (11.5,3.7) node{ \textcolor{blue}{$\mathbb{O}_{5,4}$}};
\draw (11.5,3.3) node{\tiny{$248$}};

\draw (13.5,3.7) node{ \textcolor{blue}{$\mathbb{O}_{4,5}$}};
\draw (13.5,3.3) node{\tiny{$248$}};

\draw (15.5,3.7) node{ \textcolor{SkyBlue}{$\mathbb{O}_{3,6}$}};
\draw (15.5,3.3) node{\tiny{$264$}};

\draw (17.5,3.7) node{ \textcolor{SkyBlue}{$\mathbb{O}_{2,7}$}};
\draw (17.5,3.3) node{\tiny{$264$}};

\draw (19.5,3.7) node{ \textcolor{blue}{$\mathbb{O}_{1,8}$}};
\draw (19.5,3.3) node{\tiny{$248$}};

\draw (21.5,3.7) node{   \textcolor{red}{$\mathbb{O}_{0,9}$}};
\draw (21.5,3.3) node{\tiny{$376$}};

%ligne8
\draw (4.5,4.7) node{  \textcolor{Goldenrod}{$\mathbb{O}_{8,0}$}};
\draw (4.5,4.3) node{\tiny{$56$}};

\draw (6.5,4.7) node{\textcolor{VioletRed}{$\mathbb{O}_{7,1}$}};
\draw (6.5,4.3) node{\tiny{$128$}};

\draw (8.5,4.7) node{\textcolor{SpringGreen}{$\mathbb{O}_{6,2}$}};
\draw (8.5,4.3) node{\tiny{$136$}};

\draw (10.5,4.7) node{\textcolor{VioletRed}{$\mathbb{O}_{5,3}$}};
\draw (10.5,4.3) node{\tiny{$128$}};

\draw (12.5,4.7) node{\textcolor{PineGreen}{$\mathbb{O}_{4,4}$}};
\draw (12.5,4.3) node{\tiny{$120$}};

\draw (14.5,4.7) node{\textcolor{VioletRed}{$\mathbb{O}_{3,5}$}};
\draw (14.5,4.3) node{\tiny{$128$}};

\draw (16.5,4.7) node{\textcolor{SpringGreen}{$\mathbb{O}_{2,6}$}};
\draw (16.5,4.3) node{\tiny{$136$}};

\draw (18.5,4.7) node{\textcolor{VioletRed}{$\mathbb{O}_{1,7}$}};
\draw (18.5,4.3) node{\tiny{$128$}};

\draw (20.5,4.7) node{  \textcolor{red}{$\mathbb{O}_{0,8}$}};
\draw (20.5,4.3) node{\tiny{$184$}};

%ligne7
\draw (5.5,5.7) node{   \textcolor{Goldenrod}{$\mathbb{O}_{7,0}$}};
\draw (5.5,5.3) node{\tiny{$28$}};

\draw (7.5,5.7) node{ \textcolor{SkyBlue}{$\mathbb{O}_{6,1}$}};
\draw (7.5,5.3) node{\tiny{$68$}};

\draw (9.5,5.7) node{\textcolor{SkyBlue}{ $\mathbb{O}_{5,2}$}};
\draw (9.5,5.3) node{\tiny{$68$}};

\draw (11.5,5.7) node{ \textcolor{blue}{$\mathbb{O}_{4,3}$}};
\draw (11.5,5.3) node{\tiny{$60$}};

\draw (13.5,5.7) node{ \textcolor{blue}{$\mathbb{O}_{3,4}$}};
\draw (13.5,5.3) node{\tiny{$60$}};

\draw (15.5,5.7) node{ \textcolor{SkyBlue}{$\mathbb{O}_{2,5}$}};
\draw (15.5,5.3) node{\tiny{$68$}};

\draw (17.5,5.7) node{ \textcolor{SkyBlue}{$\mathbb{O}_{1,6}$}};
\draw (17.5,5.3) node{\tiny{$68$}};

\draw (19.5,5.7) node{   \textcolor{red}{$\mathbb{O}_{0,7}$}};
\draw (19.5,5.3) node{\tiny{$92$}};

%ligne6
\draw (6.5,6.7) node{  \textcolor{Goldenrod}{$\mathbb{O}_{6,0}$}};
\draw (6.5,6.3) node{\tiny{$16$}};

\draw (8.5,6.7) node{\textcolor{SpringGreen}{$\mathbb{O}_{5,1}$}};
\draw (8.5,6.3) node{\tiny{$36$}};

\draw (10.5,6.7) node{\textcolor{VioletRed}{$\mathbb{O}_{4,2}$}};
\draw (10.5,6.3) node{\tiny{$32$}};

\draw (12.5,6.7) node{\textcolor{PineGreen}{$\mathbb{O}_{3,3}$}};
\draw (12.5,6.3) node{\tiny{$28$}};

\draw (14.5,6.7) node{\textcolor{VioletRed}{$\mathbb{O}_{2,4}$}};
\draw (14.5,6.3) node{\tiny{$32$}};

\draw (16.5,6.7) node{\textcolor{SpringGreen}{$\mathbb{O}_{1,5}$}};
\draw (16.5,6.3) node{\tiny{$36$}};

\draw (18.5,6.7) node{  \textcolor{red}{$\mathbb{O}_{0,6}$}};
\draw (18.5,6.3) node{\tiny{$48$}};

%ligne5
\draw (7.5,7.7) node{   \textcolor{Goldenrod}{$\mathbb{O}_{5,0}$}};
\draw (7.5,7.3) node{\tiny{$10$}};

\draw (9.5,7.7) node{ \textcolor{SkyBlue}{$\mathbb{O}_{4,1}$}};
\draw (9.5,7.3) node{\tiny{$18$}};

\draw (11.5,7.7) node{ \textcolor{blue}{$\mathbb{O}_{3,2}$}};
\draw (11.5,7.3) node{\tiny{$14$}};

\draw (13.5,7.7) node{ \textcolor{blue}{$\mathbb{O}_{2,3}$}};
\draw (13.5,7.3) node{\tiny{$14$}};

\draw (15.5,7.7) node{ \textcolor{SkyBlue}{$\mathbb{O}_{1,4}$}};
\draw (15.5,7.3) node{\tiny{$18$}};

\draw (17.5,7.7) node{  \textcolor{red}{ $\mathbb{O}_{0,5}$}};
\draw (17.5,7.3) node{\tiny{$26$}};

%ligne4
\draw (8.5,8.7) node{\textcolor{PineGreen}{$\mathbb{O}_{4,0}$}};
\draw (8.5,8.3) node{\tiny{$6$}};

\draw (10.5,8.7) node{\textcolor{VioletRed}{$\mathbb{O}_{3,1}$}};
\draw (10.5,8.3) node{\tiny{$8$}};

\draw (12.5,8.7) node{\textcolor{PineGreen}{$\mathbb{O}_{2,2}$}};
\draw (12.5,8.3) node{\tiny{$6$}};

\draw (14.5,8.7) node{\textcolor{VioletRed}{$\mathbb{O}_{1,3}$}};
\draw (14.5,8.3) node{\tiny{$8$}};

\draw (16.5,8.7) node{  \textcolor{red}{$\mathbb{O}_{0,4}$}};
\draw (16.5,8.3) node{\tiny{$14$}};

%ligne3
\draw (9.5,9.7) node{ \textcolor{blue}{$\mathbb{O}_{3,0}$}};
\draw (9.5,9.3) node{\tiny{$3$}};

\draw (11.5,9.7) node{ \textcolor{blue}{$\mathbb{O}_{2,1}$}};
\draw (11.5,9.3) node{\tiny{$3$}};

\draw (13.5,9.7) node{ \textcolor{blue}{$\mathbb{O}_{1,2}$}};
\draw (13.5,9.3) node{\tiny{$3$}};

\draw (15.5,9.7) node{   \textcolor{red}{$\mathbb{O}_{0,3}$}};
\draw (15.5,9.3) node{\tiny{$7$}};
\draw (12.5,-2) node{    Figure 2 : Summary table of classification of algebras $\mathbb{O}_{p,q}$  };

\end{tikzpicture}
\end{center}

  \end{landscape}
  
The table is useful to understand the symmetry and periodicity results.  
The first two rows are different than the others as in small dimensions some ``degeneracy'' occurs.
The two main properties $\bbO_{p,q}\simeq \bbO_{q,p}\;$ 
and $\bbO_{p,q+4}\simeq \bbO_{p+4,q}\;$ are 
symmetry with respect to the vertical middle axis and a periodicity on each row of the table,
respectively. The algebras, $\bbO_{n,0}$ and $\bbO_{0,n}$, are exceptional. 
On each row, besides the two exceptional algebras, there are exactly two other non-isomorphic algebras in the case $n$ is odd, and exactly three in the case $n$ is even ($n\geq 5$).

\medskip

\noindent \textbf{Acknowledgments}.
The authors are grateful to Valentin Ovsienko for enlightenning discussions.   
M.K. was partially supported by the 
Interuniversity Attraction Poles Programme initiated by the Belgian Science Policy Office. 
S.M-G. was partially supported by the PICS05974 ``PENTAFRIZ'' of CNRS.
M.K., resp. S.M-G., would like to thank the Institute Camille Jordan, CNRS, Universit\'e Lyon 1,
resp. l'Universit\'e de Li\`ege, for hospitality.
\vskip 1cm

%%%%%%%%%%%%%%%%%%%%%%%%%%%%%%%%%%%%%%%%%%%%%%%%%%
%%%%%%%%%%%%%%%%%%%%%%%%%%%%%%%%%%%%%%%%%%%%%%%%%%
% Bibliographie
%%%%%%%%%%%%%%%%%%%%%%%%%%%%%%%%%%%%%%%%%%%%%%%%%%
%%%%%%%%%%%%%%%%%%%%%%%%%%%%%%%%%%%%%%%%%%%%%%%%%%

%\bibliographystyle{abbrv}
%\bibliography{biblio}

%\begin{thebibliography}{99}
%%%%%%%%%%%%%%%%%%%%%%%%%%%%%%%%%%%%%%%%%%%%%%%%%%
%\bibitem{AE} H. Albuquerque, A. Elduque, J.M. P\'erez-Izquierdo, {\it Alternative quasialgebras}, Bull. Austral. Math. Soc.  {\bf 63}  (2001), 257--268.

%\bibitem{AM1999} H. Albuquerque, S. Majid, {\it Quasialgebra structure of the octonions}, J. Algebra {\bf 220} (1999), 188--224.

%\bibitem{AM2002} H. Albuquerque, S. Majid, {\it Clifford algebras obtained by twisting of group algebras}, J. Pure Appl. Algebra  {\bf 171}  (2002), 133--148.

%\bibitem{B2002} Baez J., "The octonions" {\it Bulletin (new series) of the American Mathematical Society } \textbf{39}, no. 2 (2002) : 145 -- 205.

%\bibitem{LMGO2011} A. Lenzhen, S. Morier-Genoud, V. Ovsienko,  {\it New solutions to the Hurwitz problem on square identities}, J. Pure Appl. Algebra {\bf 215} (2011), 2903--2911.

%\bibitem{MGO2013} S. Morier-Genoud, V. Ovsienko, {\it Orthogonal designs and a cubic binary function},  IEEE Trans. Information Theory, {\bf 59}:3 (2013) 1583--1589.

%\bibitem{MGO2011} S. Morier-Genoud, V. Ovsienko, {\it A series of algebras generalizing the octonions and Hurwitz-Radon identity}, Comm. Math. Phys. {\bf 306} (2011), 83--118.

%\end{thebibliography}

%\nocite{*}
%\thispagestyle{empty}

\end{document}